\documentclass[a4paper,11pt,reqno]{amsart}
\usepackage{latexsym,amscd,amssymb,url}
\usepackage{comment}
\usepackage{extarrows}
\usepackage[all,cmtip]{xy}  
\usepackage{graphicx}
\usepackage{xcolor}
\usepackage{enumitem} 
\definecolor{dblue}{rgb}{0,0,.6}
\usepackage[colorlinks=true, linkcolor=dblue, citecolor=dblue, filecolor = dblue, menucolor = dblue, urlcolor = dblue]{hyperref}

\numberwithin{equation}{section}

\newtheorem{theorem}{Theorem}[section]

\theoremstyle{plain}

\newtheorem{lemma}[theorem]{Lemma}

\newtheorem{remark}[theorem]{Remark}

\newcommand{\A}{\mathbb A}

\newcommand{\CP}{\mathbb P}

\newcommand{\im}{\operatorname{im}}

\newcommand{\Sym}{\operatorname{Sym}}
\newcommand{\id}{\operatorname{id}}

\newcommand{\Spec}{\operatorname{Spec}}

\newcommand{\dashedlongrightarrow}{\xymatrix@1@=15pt{\ar@{-->}[r]&}}
\renewcommand{\longrightarrow}{\xymatrix@1@=15pt{\ar[r]&}}
\renewcommand{\mapsto}{\xymatrix@1@=15pt{\ar@{|->}[r]&}}
\renewcommand{\twoheadrightarrow}{\xymatrix@1@=15pt{\ar@{->>}[r]&}}
\newcommand{\hooklongrightarrow}{\xymatrix@1@=15pt{\ar@{^(->}[r]&}}
\newcommand{\congpf}{\xymatrix@1@=15pt{\ar[r]^-\sim&}}
\renewcommand{\cong}{\simeq}

\begin{document}    
\title[Geometric retract rationality of norm varieties]{Geometric retract rationality of norm varieties} 

\author{Stefan Schreieder} 
\address{Institute of Algebraic Geometry, Leibniz University Hannover, Welfengarten 1, 30167 Hannover, Germany.}
\email{schreieder@math.uni-hannover.de}

\date{June 9, 2023} 
%



\begin{abstract}
Let $k$ be a field of characteristic zero and let  $\alpha=(a_1,\dots ,a_n)$ be any ordered sequence of invertible elements $a_i \in k^\ast$. 
We show that for any prime $\ell$,  the associated norm variety $X_{\alpha,\ell}$   is geometrically retract rational. 
This generalizes a recent result in \cite{BHS}, where geometric $\A^1$-connectedness (an a priori weaker notion) had been proven. 
\end{abstract}
 
\maketitle 

\section{Introduction}
 
Let $\ell$ be a prime and let $k$ be field of characteristic zero.
Let $\alpha:=(a_1,\dots ,a_n)$ be an ordered sequence of elements $a_1,\dots ,a_n\in k^\ast$.
Associated to this there is a standard norm variety $X_{\alpha,\ell}$ (see \cite{rost02,SJ}) which played a crucial role in Voevodsky's proof of the Bloch--Kato conjecture \cite{Voe}.

 Asok proved that $X_{\alpha,\ell}$  is geometrically rationally connected, see \cite[Proposition 2.6 and Remark 2.9]{asok}.
In \cite[Theorem 1]{BHS}, the authors improve Asok's result by showing that  $X_{\alpha,\ell} $ is geometrically $\A^1$-connected.
Equivalently, $X_{\alpha,\ell} $  becomes universally $R$-trivial over the algebraic closure $\bar k$, which means that after any field extension $L/\bar k$, any two $L$-rational points of $X_{\alpha,\ell}\times_kL$ can be joined by a chain of rational curves.
(The result of Asok handled the case where $L$ is algebraically closed.)

Recall that a variety $X$ is retract rational if the identity on $X$ factors as rational map through some projective space.
Any retract rational variety is unirational, while the converse is not true in general, see e.g.\ \cite{AM}.
Moreover, 
Asok and Morel showed that retract rational varieties are $\A^1$-connected (at least in characteristic zero), see \cite[Theorem 2.3.6]{asok-morel}.
The converse to this is an open question, see \cite[Remark 2.3.11]{asok-morel}.
It is thus natural to ask whether standard norm varieties are not only geometrically $\A^1$-connected but in fact geometrically retract rational.
This question goes back to Asok \cite[Remark 2.9]{asok} and  Balwe--Hogadi--Sawant \cite[Introduction]{BHS}. 
In this short note we answer this question positively. 

\begin{theorem} \label{thm:main}
Let $k$ be a field of characteristic zero and let $\alpha:=(a_1,\dots ,a_n)$ be an ordered sequence  of elements  $a_1,\dots ,a_n\in k^\ast$ with $n\geq 2$.
Then for any prime $\ell$, the standard norm variety $X_{\alpha,\ell}$ is geometrically retract rational.
\end{theorem}

The above theorem is formulated for fields of characteristic zero, because the standard norm varieties under consideration are mostly studied under this assumption, see e.g.\ \cite{SJ}.
However,  our proof does not make use of resolution of singularities and works in arbitrary characteristic. 
In particular, standard norm varieties in positive characteristic, defined analogously to the case of characteristic zero,  are geometrically retract rational, despite the fact that it is not clear that these varieties have a smooth projective model.

We note that  $X_{\alpha,\ell}$ is in general not retract rational over the ground field $k$.
In fact, whenever the symbol associated to $\alpha$ is nonzero in $K^M_n(k)/\ell$, i.e.\ in  Milnor K-theory modulo $\ell$,  then the norm variety $X_{\alpha,\ell}$ is never retract rational (it is not even unirational, because it does not contain a rational point).

\subsection{Notation} \label{subsec:notation}
An algebraic scheme is a separated scheme of finite type over a field.
A $k$-variety is an integral algebraic scheme over a field $k$.
For a $k$-variety $X$ and a positive integer $m$, we denote by $\Sym ^mX=X^m/S_m$ the $m$-th symmetric power of $X$.
In general, $\Sym^mX$ is an algebraic scheme over $k$; it is integral, hence a variety in the above sense, if $X$ is geometrically integral.
We will use that $\Sym^m$ is a covariant functor on the category of geometrically integral $k$-varieties.
 
A $k$-variety $X$ is retract rational if there is a rational map $f:X\dashrightarrow \CP^N$ for some $N\geq 1$ and a rational map $g:\CP^N\dashrightarrow X$ such that $g\circ f$ is defined and agrees with the identity as a rational map.
In other words, there is a rational map $f:X\dashrightarrow \CP^N$ which admits a rational section.
Clearly, if $X$ is retract rational, then so is any birational model (e.g.\ any non-empty open subset).
Moreover, $X$ is retract rational if and only if $X\times \CP^n$ is retract rational for some $n\geq 0$.
In particular, any stably rational variety is retract rational. 

By definition,  $X$ is retract rational if and only if  there are open dense subsets $U\subset X$ and $V\subset \CP^N$ for some $N$ such that the identity on $U$ factors through $V$: $U\to V\to U$. 
It  follows from this description that  retract rational varieties are geometrically integral (it is in the above notation enough to check that $U$ is geometrically integral, which follows from the fact that $V$ is geometrically integral).
In particular, if $X$ is a retract rational $k$-variety then any symmetric power $\Sym ^m X$ is a $k$-variety in the above sense.

\section{Three lemmas} 

We say that a rational map $f:X\dashrightarrow Y$ between two $k$-varieties is generically injective if there are dense open subsets $U\subset X$ and $V\subset Y$ such that $f$ restricts to a closed embedding $U\hookrightarrow V$.
Equivalently, if $\eta_X\in X$ denotes the generic point, then $f$ induces an isomorphism between residue fields $\kappa(f(\eta_X))\stackrel{\cong}\to \kappa(\eta_X)$. 

\begin{lemma} \label{lem:0}
Let $k$ be a field and let $Y$ be a retract rational $k$-variety.
Then for any rational map $f:Y\dashrightarrow \CP^N$ which is generically injective, there is a rational section $g:\CP^N\dashrightarrow Y$, i.e.\ a rational map $g$ such that $g\circ f$ is defined and coincides with the identity as a rational map. 
\end{lemma}
\begin{proof}
By Saltman's criterion \cite[Theorem 3.4]{saltman} (see also \cite[Proposition 1.2]{CTS}), there is a dense open subset $V\subset Y$ such that for any local $k$-algebra $A$ with residue field $\kappa$, the map $V(A)\to V(\kappa)$ is surjective.
(In fact, this is the easy direction of Saltman's criterion: choose $V$ such that there is a dense open subset $U\subset \A^N$ such that the identity on $V$ factors via $V\to U\to V$.
Any map $\epsilon:\Spec \kappa\to \A^N$ lifts to a map $\epsilon_A:\Spec A\to \A^N$.
If $\im(\epsilon)\subset U$, then $\im (\epsilon_A)\subset U$, because $A$ is local.
Hence,   $U(A)\to U(\kappa)$ is surjective.
By the factorization $V\to U\to V$ of $\id_V$, $V(A)\to V(\kappa)$ is surjective as well.)

We apply this to the local ring $A:=\mathcal O_{\CP^N,f(\eta_Y)}$ of $\CP^N$ at the image of the generic point of $Y$.
Since $f$ is generically injective, the residue field of $A$ is isomorphic to $k(Y)$.
It follows that the ``diagonal'' point $\delta_Y\in Y_{k(Y)}$, corresponding to the canonical map $\Spec k(Y)\to Y$,  lifts to an $A$-valued point of $Y$.
In other words, there is a morphism $\Spec A\to Y$ which restricts to the inclusion of the generic point $\eta_Y$ on the closed point of $\Spec A$.
Since $A=\mathcal O_{\CP^N,f(\eta_Y)}$, we conclude that there is a rational map $g:\CP^N\dashrightarrow Y$ which is defined at $f(\eta_Y)$ and such that $g\circ f$ coincides as rational map with the identity on $Y$.
This concludes the proof of the lemma.
\end{proof}

\begin{remark} 
Since any variety is birational to a hypersurface, it follows from Lemma \ref{lem:0} that any retract rational k-variety $Y$ admits a rational map $Y\dashrightarrow \CP^{\dim Y+1}$ with a rational section.
\end{remark}

\begin{remark} 
 The argument in Lemma \ref{lem:0}
 works verbatim if we replace $\CP^N$ by any $k$-variety $Z$. 
Hence, if $Y$ is retract rational, then any generically injective map $f:Y\dashrightarrow Z$ to a $k$-variety $Z$ admits a rational section.
 (Thanks to Remy van Dobben de Bruyn for pointing this out.)
In fact, this a priori stronger statement is   a formal consequence of Lemma \ref{lem:0} because any $k$-variety $Z$ admits a generically injective rational map $\iota:Z\dashrightarrow \CP^N$ and so we may apply Lemma \ref{lem:0} to $\iota\circ f$ to produce a rational section of $\iota\circ f$ which restricts to a rational section of $f$, as we want.
 \end{remark}

\begin{lemma} \label{lem:1}
Let $Y$ be a retract rational $k$-variety over an infinite field $k$.
Then $\Sym^mY$ is retract rational over $k$ for any $m\geq 1$.
\end{lemma}
\begin{proof} 
The statement is trivial if $\dim Y=0$, because retract rational varieties are geometrically integral and so $Y=\Spec k$ in this case.
We may therefore assume that $Y$ has positive dimension.
Up to replacing $Y$ by a projective model, we may assume that there is a closed embedding $f:Y\hookrightarrow  \CP^{N}$ for some $N$. 
Since $\dim Y\geq 1$,  composing the embedding $Y\hookrightarrow \CP^N$ with some Veronese embedding of $\CP^N$, we can ensure that
\begin{align}\label{eq:lem:1}
\dim   \left\langle f(Y)\right\rangle \geq m,
\end{align}
where $\left\langle f(Y)\right\rangle \subset \CP^N$ denotes the smallest linear subspace that contains $f(Y)$.

Since $Y$ is retract rational,  Lemma \ref{lem:0} shows that $f$ admits a rational section, i.e.\ a rational map $g:\CP^N\dashrightarrow Y$ such that $g\circ f$ is defined and coincides with the identity (as a rational map). 
Taking the symmetric product, we get maps
$$
\Sym^m f:\Sym^m Y\longrightarrow \Sym^m\CP^N\ \ \text{and}\ \ \Sym^mg:\Sym^m \CP^N\dashrightarrow \Sym^m Y
$$
such that $\Sym^m g\circ \Sym^m f$ is defined and coincides with the identity.

By \cite{mattuck}, there is a birational map
$$
\varphi: \Sym^m\CP^N\stackrel{\cong}\dashrightarrow \CP^{m N} .
$$
Hence there are open dense subsets $U\subset \Sym^m\CP^N$ and $V\subset \CP^{mN}$ such that $\varphi$ induces an isomorphism $U\cong V$.

From now on we use that $k$ is infinite, which implies that the $k$-rational points of $U$ are Zariski dense.
We may therefore pick a general $k$-rational point $\xi\in U$ and note that $\xi$ corresponds to a set $\{x_1,\dots ,x_m\}$ of general $k$-rational points $x_i\in \CP^N$.
By (\ref{eq:lem:1}), $\dim \left\langle f(Y)\right\rangle \geq m$ and so $N\geq m$ (because $\left\langle f(Y)\right\rangle $ is a subspace of $\CP^N$).
It follows that the general points $x_1,\dots ,x_m$ of $\CP^N$ are linearly independent in $\CP^N$.
Similarly, a general $k$-rational point $\zeta\in \Sym^mY$ corresponds to a set $\{y_1,\dots ,y_m\}$ of general $k$-rational points $y_i\in Y$.
(Here we use that the set of $k$-rational points of $Y$ is Zariski dense because $k$ is infinite and $Y$ is retract rational, hence unirational over $k$.)
Since $\zeta$ is general, $f$ is defined at each of the points $y_1,\dots ,y_m$ and since $\dim \left\langle f(Y)\right\rangle \geq m$, the points $f(y_1),\dots ,f(y_m)$ of $\CP^N$ are linearly independent.
It follows that there is a linear automorphism $h:\CP^N\stackrel{\cong}\to \CP^N$ defined over $k$ such that $h(f(y_i))=x_i$ for all $i$.
The induced automorphism
$$
\Sym^m h:\Sym^m\CP^N\longrightarrow \Sym^m\CP^N
$$
satisfies
$$
\Sym^m h(\Sym^m f(\zeta))=\xi\in U.
$$
It follows that the composition
$$
\Sym^m Y\stackrel{\Sym^m f}\dashrightarrow \Sym^m\CP^N \stackrel{\Sym^m h} \longrightarrow \Sym^m\CP^N \stackrel{\varphi}\dashrightarrow \CP^{m N}
$$
is defined and admits a rational section, given by
$$
\CP^{m N}\stackrel{\varphi^{-1}}\dashrightarrow \Sym^m\CP^N  \stackrel{\Sym^m h^{-1}} \longrightarrow \Sym^m\CP^N\stackrel{\Sym^m g}\dashrightarrow \Sym^m Y .
$$
(Here we use that $\Sym^m g$ is defined at $\Sym^m f(\zeta)$, because $\zeta\in \Sym^mY$ was chosen to be general and $g\circ f$ is defined as a rational map by assumption.)
This concludes the proof of the lemma.
\end{proof}

\begin{lemma} \label{lem:2}
Let $h:X\to Y$ be a dominant morphism of $k$-varieties whose generic fibre $X_\eta$ is irreducible.
If $X_\eta$ is retract rational over $k(Y)$ and $Y$ is retract rational over $k$, then $X$ is retract rational over $k$.
\end{lemma}
\begin{proof}
Since the generic fibre $X_\eta$ of $h$ is retract rational, there is a rational map
$$
f:X_\eta\dashrightarrow \CP^r_{k(Y)}
$$
which admits a rational section $g:\CP^r_{k(Y)}\dashrightarrow X_\eta$.
The maps $f$ and $g$ spread out to rational maps of $k$-varieties
$$
F:X\dashrightarrow \CP^r\times Y\ \ \text{and}\ \ G:\CP^r\times Y \dashrightarrow X ,
$$
such that $G\circ F$ is defined and agrees with the identity as a rational map.
Moreover, $F$ and $G$ are rational maps over $Y$, i.e.\ they are compatible with the corresponding maps to $Y$.
Since $Y$ is retract rational as a $k$-variety, there is a rational map $f':Y\dashrightarrow \CP^N$ with a rational section $g':\CP^N\dashrightarrow Y$.
We identify $ \CP^r\times \CP^N$ birationally with $\CP^{r+N}$, by noting that both varieties are compactifications of $\A^{r+N}$.
The compositions
$$
(\id\times f')\circ F:X\dashrightarrow \CP^{r+N}
$$
and
$$
G\circ (\id\times g'):\CP^{r+N}\dashrightarrow X
$$
are then rational maps of $k$-varieties such that the composition
$$
G\circ (\id\times g')\circ (\id\times f')\circ F:X\dashrightarrow X
$$
is defined (this uses that $F$ and $G$ are rational maps over $Y$, hence defined at some point that dominates the generic point of $Y$) and coincides with the identity.
It follows that $X$ is retract rational, as claimed.
\end{proof}

\section{Proof of Theorem \ref{thm:main}}

We may replace $k$ by $\bar k$ and assume that $k$ is algebraically closed.
For $n=2$, $X_{\alpha,\ell}$ is a Severi--Brauer variety and so it is rational over the algebraically closed field $k$.
We may thus assume $n\geq 3$ and argue by induction.
By construction (see e.g.\ \cite[\S 2]{SJ} or \cite[\S 4]{BHS}), there is a birational model $X$ of $X_{\alpha,\ell}$ with  a dominant morphism $f:X\to \Sym^\ell Y$ such that
\begin{itemize}
\item $Y$ is birational to $X_{\alpha',\ell}$ with  $\alpha':=(a_1,\dots ,a_{n-1})$;
\item the generic fibre of $f$ is birational to the norm hypersurface $\{N-a_n=0\} $ over $k(\Sym^\ell Y)$, where 
$$
N:R_{L/K}\mathbb G_{m,L}\longrightarrow \mathbb G_{m, K}
$$ 
is the map of tori associated to the norm map of fields $N:L\to K$, where $L:=k(Y\times \Sym^{\ell-1}Y)$ and $K:=k(\Sym^\ell Y)$, and  $ R_{L/K}\mathbb G_{m,L}$ denotes the Weil restriction of $\mathbb G_{m,L}$ to $K$.
\end{itemize} 
Here the models $X$ and $Y$ are not necessarily proper nor smooth.
(Of course we could always assume  that one of the two conditions is satisfied, but the argument does not need that.)
We could assume that the generic fibre of $f$ is actually isomorphic to the norm hypersurface  $\{N-a_n=0\} $ over $k(\Sym^\ell Y)$, but again this will not play any role in the argument.

Since $k=\bar k$, there is an element $\xi\in k$ with $\xi^\ell=a_n$.
We may regard $\xi$ as an element of $L=k(Y\times \Sym^{\ell-1}Y)$ and we note that $N(\xi)=a_n$.
Multiplication with $\xi$ induces an automorphism of $R_{L/K}\mathbb G_{m,L}$ which identifies the norm 1 torus $\{N-1=0\}$ with   $\{N-a_n=0\}$.
Hence, since $\{N-1=0\}$  is retract rational by \cite[Theorem 4.1]{endo},  $\{N-a_n=0\}$ is retract rational as well.
Theorem \ref{thm:main} therefore follows from Lemmas \ref{lem:1} and \ref{lem:2} above, where we use that $k$ is algebraically closed, hence an infinite field.

\begin{remark}
Stably rational varieties are retract rational and the converse is known to fail  in general over non-closed fields.
However, it is not known whether the two notions coincide over algebraically closed fields, cf.\  \cite{PS}.
While Lemma \ref{lem:2} remains true if we replace retract rationality by stable rationality, this does not seem to be  
clear for 
 Lemma \ref{lem:1}.
More importantly,  the norm 1 tori over $k(\Sym^\ell Y)$ that appear in the above argument are retract rational for all primes,  but they are not stably rational for primes $\ell\geq 5$,  see \cite[Theorem 4.1]{endo}.
In particular,   the norm varieties $X_{\alpha,\ell}$ for $\ell\geq 5$ yield by Theorem \ref{thm:main} natural candidates for varieties over algebraically closed fields that are retract rational but which may potentially be not stably rational in general.
\end{remark}

\section*{Acknowledgements}    
Thanks to Remy van Dobben de Bruyn and the referee for comments, and to Anand Sawant for explaining the results in \cite{BHS} to me. 
This project has received funding from the European Research Council (ERC) under the European Union's Horizon 2020 research and innovation programme under grant agreement No 948066 (ERC-StG RationAlgic).


\end{document}